\documentclass[11pt]{amsart}

\usepackage{epigamath}

\usepackage[english]{babel}

\numberwithin{equation}{section}

\usepackage{amscd, amsmath, amssymb, amsthm}
\usepackage[utf8]{inputenc}
\usepackage[all,cmtip]{xy}

\usepackage{graphicx}
\usepackage{tikz-cd}
\usepackage{tikz}
\usetikzlibrary{calc} % shapes,arrows,shapes.multipart,positioning, fit, were not used
\usetikzlibrary{arrows.meta,bending,decorations.markings,intersections} %< added

\tikzset{% inspired by https://tex.stackexchange.com/a/316050/121799
    arc arrow/.style args={%
    to pos #1 with length #2}{
    decoration={
        markings,
         mark=at position 0 with {\pgfextra{%
         \pgfmathsetmacro{\tmpArrowTime}{#2/(\pgfdecoratedpathlength)}
         \xdef\tmpArrowTime{\tmpArrowTime}}},
        mark=at position {#1-\tmpArrowTime} with {\coordinate(@1);},
        mark=at position {#1-2*\tmpArrowTime/3} with {\coordinate(@2);},
        mark=at position {#1-\tmpArrowTime/3} with {\coordinate(@3);},
        mark=at position {#1} with {\coordinate(@4);
        \draw[-{Stealth[length=#2,bend]}]       
        (@1) .. controls (@2) and (@3) .. (@4);},
        },
     postaction=decorate,
     }
}

\newtheorem{thm}{Theorem}[section]
\newtheorem{lem}[thm]{Lemma}
\newtheorem{prop}[thm]{Proposition}
\newtheorem{cor}[thm]{Corollary}

\theoremstyle{remark}
\newtheorem{rem}[thm]{Remark}

\theoremstyle{definition}
\newtheorem{defn}[thm]{Definition}
\newtheorem{exmp}[thm]{Example}

\newtheoremstyle{case}{}{}{}{}{}{:}{ }{}
\theoremstyle{case}

\DeclareMathOperator{\Bl}{Bl}

\DeclareMathOperator{\Sing}{Sing}
\DeclareMathOperator{\Aut}{Aut}

\DeclareMathOperator{\Ceva}{Ceva}

\DeclareMathOperator{\Tors}{Tors}

\newcommand{\abs}[1]{\left\vert#1\right\vert}
\newcommand{\A}{\mathscr{A}}

\EpigaVolumeYear{5}{2021}
\EpigaArticleNr{15}
\ReceivedOn{April 23, 2020}
\AcceptedOn{September 1, 2021}

\title{The fundamental group of quotients of products of some topological spaces by a finite group --- A generalization of a theorem of Bauer--Catanese--Grunewald--Pignatelli}
\titlemark{Fundamental groups of quotients of products by finite groups}

\author{Rodolfo Aguilar Aguilar}
\address{Universit\'e Grenoble-Alpes, Institut Fourier, 100 rue de Maths, 384610, Gières, France.}
\email{rodolfo.aguilar-aguilar@univ-grenoble-alpes.fr}

\authormark{R. Aguilar Aguilar}

\AbstractInEnglish{We provide a description of the fundamental group of the quotient of a product of topological spaces $X_i$, each admitting a universal cover, by a finite group $G$, provided that there is only a finite number of path-connected components in $X_i^g$ for every $g\in G$. This generalizes previous work of Bauer--Catanese--Grunewald--Pignatelli and Dedieu--Perroni.}

\MSCclass{14F35; 54B15; 18A32; 20E34}

\KeyWords{fundamental group; quotients by finite group; orbifolds}

\TitleInFrench{Le groupe fondamental de quotients de produits de certains espaces topologiques par un groupe fini --- Généralisation d'un théorème de Bauer--Catanese--Grunewald--Pignatelli}

\AbstractInFrench{Nous fournissons une description du groupe fondamental du quotient d'un produit d'espaces topologiques $X_i$, chacun admettant un revêtement universel, par un groupe fini $G$, pourvu qu'il n'existe qu'un nombre fini de composantes connexes par arcs dans $X_i^g$ pour chaque $g\in G$. Cela généralise des résultats antérieurs de Bauer--Catanese--Grunewald--Pignatelli et de Dedieu--Perroni.}

\acknowledgement{Partially supported by ERC ALKAGE, grant No. 670846 and ANR project Hodgefun ANR-16-CE40-0011-01.}

\begin{document}

%%%%%%%%%%%%%%%%%%%%%%%%%%%%%%%
% Title page
%%%%%%%%%%%%%%%%%%%%%%%%%%%%%%%

%\removeabove{}
%\removebetween{}
%\removebelow{}

\maketitle

\begin{prelims}

\DisplayAbstractInEnglish

\bigskip

\DisplayKeyWords

\medskip

\DisplayMSCclass

\bigskip

\languagesection{Fran\c{c}ais}

\bigskip

\DisplayTitleInFrench

\medskip

\DisplayAbstractInFrench

\end{prelims}

%%%%%%%%%%%%%%%%%%%%%
% Table of Contents
%%%%%%%%%%%%%%%%%%%%%

\newpage

\setcounter{tocdepth}{1} 

\tableofcontents

%%%%%%%%%%%%%%%%%%%%%
% Content begins here
%%%%%%%%%%%%%%%%%%%%%

\section{Introduction}
The fundamental group of a quotient of a Haussdorf space $X$ by a finite group $G$ acting freely can be computed noticing that $X\to X/G$ is a covering map, and then using the long exact sequence of homotopy groups associated to a fibration. When $X=X_1\times\cdots\times X_k$ and $G$ acts on each $X_i$ freely and diagonally on $X$, the fundamental group of $X_1\times\cdots\times X_k$ sits as a finite-index normal subgroup of $\pi_1(X/G)$.

In the case where each $X_i$ is a projective smooth curve and the action of $G$ is only \emph{faithful}, the following Theorem was shown in \cite{CataneseBauer1}.
\begin{thm}\cite[Theorems~0.10 and~4.1]{CataneseBauer1}\label{thm:Int1} Let $C_1,\ldots,C_k$ be smooth projective curves and let $G$ be a finite group acting faithfully by automorphisms on each of them. Consider the diagonal action of $G$ on the product $C_1\times \cdots\times C_k$, then the fundamental group of $(C_1\times \cdots \times C_k)/G$ admits a normal finite index subgroup $\mathcal{N}$ isomorphic to a product of fundamental groups of smooth projective curves. 
\end{thm}

It was later extended in \cite{dedieu} to the case when the action of $G$ is non-necessarily faithful. There, they quotient $G$ to obtain a group acting faithfully, follow the subsequent arguments and then extend again to $G$. 

Let us explain briefly the method of proof in \cite{CataneseBauer1}.  First, they consider the orbifold surface groups $T_i$ of $C_i/G$, which are an extension of $G$ by $\pi_1(C_i)$ and hence come with a surjective morphism to $G$ (see Subsection \ref{ss:FundamentalQuoProd}). They show that the fundamental group $\pi_1((C_1\times\cdots\times C_k)/G)$ is isomorphic to the quotient of the fiber product $\mathbb{H}:=T_1\times_G \cdots \times_G T_k$ by the normal subgroup $\Tors(\mathbb{H})$ generated by the elements of torsion. 

The second part relies on the following proposition whose proof uses abstract group theoretic arguments.
\begin{prop}\cite[Proposition~3.5]{dedieu}\label{prop:Int1} There exists a short exact sequence of groups 
$$1 \to E \to \mathbb{H}/\Tors{\mathbb{H}}\to T\to 1 $$
where $E$ is finite and $T$ is a group of finite index in a product of orbifold surfaces groups.
\end{prop}

They finally use Proposition \ref{prop:Int1} and properties of the orbifold surface groups such as residually finiteness and cohomological goodness to construct a subgroup of $\mathbb{H}/\Tors \mathbb{H}$ intersecting $E$ trivially and satisfying the required properties.

Here, a more geometric approach is used via fundamental groups of stacks or orbispaces \cite{2005noohi,chen2001homotopy}. This theory permits to see $X\to [X/G]$ as a covering map under some conditions on $X$, where $[X/G]$ denotes the quotient stack, and a long exact sequence of homotopy groups is available. We will denote the fundamental group of the stack $[X/G]$ by $\pi_1([X/G])$.

For $i=1,\ldots,k$, let $X_i$ be a connected, locally path-connected and semi-locally simply connected topological space with an action of a finite group $G$, consider the diagonal action of $G$ in $X:=X_1\times \cdots\times X_k$ and denote by $I$ the subgroup generated by the elements having a fixed point in every $X_i$ for $i=1,\ldots,k$.
We can formulate now our first main Theorem.

\begin{thm}\label{thm:Int2} Let $X, X_1,\ldots,X_k$ and $G$ as above.  Suppose that the number of path connected components in  the fixed locus set $X_i^g$ of an element $g\in G$ is finite for every $g\in G$ and $i=1,\ldots,k$. Then there exists a homomorphism
$$\pi_1(X/G)\to \prod_{i=1}^k \pi_1\left([(X_i/I)\left/ (G/I)\right.]\right)$$
which has finite kernel and its image is a finite-index subgroup.
\end{thm}

This can be seen as a generalization of Proposition \ref{prop:Int1} (Bauer--Catanese--Grunewald--Pignatelli) by the remarks preceding the statement of the Proposition.

The action of $G/I$ on $X_i/I$ is induced by the action of $G$ in $X_i$. Note that if $k=1$ then $G/I$ can be seen to act freely on $X_1/I$ and $\pi_1([(X_1/I)/(G/I)])=\pi_1((X_1/I)/(G/I))$ but $(X_1/I)/(G/I)\cong X_1/G$. The same argument works if we make the product of the same topological space, which gives the following corollary.

\begin{cor}\label{cor:Int1} Let $X_i=X_1$ for $i=2,\ldots,k$ and $G$ satisfy the hypothesis of the above theorem. Then the homomorphism $\pi_1(X/G)\to \pi_1(X_1/G)^k$ has finite kernel and its image is a finite-index subgroup.
\end{cor}
An important case of Theorem \ref{thm:Int2} and Corollary \ref{cor:Int1} is when $X_i$ is a smooth complex algebraic variety for $i=1,\ldots, k$. Indeed, the fundamental group of a variety with quotient singularities is the fundamental group of a smooth variety.

Our second main Theorem generalizes Theorem \ref{thm:Int1} (Bauer--Catanese--Grunewald--Pignatelli). It can be stated without using the language of stacks or orbispaces.
\begin{thm} \label{thm:SecondMain}
Let $X, X_1,\ldots X_k$ and $G$ satisfy the hypothesis of Theorem \ref{thm:Int2}. Suppose that $\pi_1(X/G)$ is residually finite. Then there exists a normal finite-index subgroup $\mathcal{N}\lhd\pi_1(X/G)$ isomorphic to a product:
\[\mathcal{N}\cong \prod_{i=1}^k H_i.\]
with $H_i\lhd \pi_1(X_i/I)$ normal finite-index subgroups.
\end{thm}

As a corollary, following closely the arguments used in \cite{CataneseBauer1}, we show that for smooth curves $C_1,\ldots, C_k$ the group $\pi_1((C_1\times \cdots\times
 C_k)/G)$ is residually finite. Hence, we have that Theorem \ref{thm:Int1} (Bauer--Catanese--Grunewald--Pignatelli, Dedieu--Perroni) is valid in the case when the curve is non-necessarily compact. 

\begin{cor}\label{cor:Int3} Let $C_1,\ldots,C_k$ be smooth algebraic curves and let $G$ be a finite group acting on each of them. Then there exists a normal subgroup $\mathcal{N}\lhd\pi_1((C_1\times \cdots\times C_k)/G)$ of finite index, isomorphic to a product $\Pi_1\times \cdots\times \Pi_k$, where $\Pi_j$ is either the fundamental group of a smooth projective curve or a free group of finite rank.
\end{cor}

The paper is organized as follows: in Section \ref{section:Preliminaries} preliminary results are given. Then the first main Theorem is proved in Section \ref{section:MainTheorem} and the proof of the second main Theorem together with some applications are given in Section \ref{section:Applications}.

\section{Preliminaries}\label{section:Preliminaries}
\subsection{Properties of fundamental group of topological stacks}\label{ss:TopologicalStacks}
Let $X$ be a connected, semi-locally simply connected and locally path-connected topological space and $G$ a finite group acting continuously on it. 
\subsubsection{Fiber homotopy exact sequence}\label{ss:FiberHomotopyES} There exists a homotopy theory for stacks and the existence of the long exact sequence of homotopy, see \cite{NOOHI2014612}, is more general than what follows, however we only need the following case: consider the topological stack $\mathcal{X}=[X/G]$, a point $x\in X$ and denote by $\bar{x}\in \mathcal{X}$ the image of $x$. We have an associated fibration $G\to X\to \mathcal{X} $ and a long exact sequence of homotopy groups, 
$$\cdots\to\pi_{n+1}(\mathcal{X},\bar{x})\to \pi_n(G,Id)\to \pi_n(X,x)\to \pi_n (\mathcal{X},\bar{x}) \to \pi_{n-1}(G,\mbox{Id})\to\cdots $$
the map $\pi_n(G,Id)\to \pi_n(X,x)$ is induced by the orbit $G\cdot x \hookrightarrow X$.

\subsubsection{Action on the universal cover}
The hypothesis made on $X$ ensures that there exists an universal cover $\tilde{X}$ and moreover, if we let $\mathcal{X}=[X/G]$ as in \S \ref{ss:FiberHomotopyES}, we have an action of $\pi_1(\mathcal{X},\bar{x})$ on $\tilde{X}$ (see \S \ref{ss:DescriptionOfOrbispaces}). We will use several times the following lemma in what follows.
\begin{lem}\label{lem:InertiaIsInjective} Consider the action of $\pi_1(\mathcal{X},\bar{x})$ in $\tilde{X}$, let $y\in \tilde{X}$ and denote by $I_y$ the isotropy group of the action. Then there exists a monomorphism $I_y\to G$.
\end{lem}
\begin{proof}
By \S \ref{ss:FiberHomotopyES} we obtain a short exact sequence 
$$1\to \pi_1(X,x)\to \pi_1(\mathcal{X},\bar{x})\to G\to 1, $$ 
as the action of $\pi_1(X,x)$ on $\tilde{X}$ is free, we obtain that the restriction of $\pi_1(\mathcal{X},\bar{x})\to G$ to $I_y$ is injective.
\end{proof}

\subsection{Product of topological spaces}
\label{subsection:Product of Curves}
\subsubsection{Fundamental group of the quotient of a product}\label{ss:FundamentalQuoProd}

For $i=1,\ldots,k$ let $X_i$ as in \S \ref{ss:TopologicalStacks} be a connected, semi-locally simply connected and locally path-connected topological space and $G$ a finite group acting on each of them. 

By \S \ref{ss:FiberHomotopyES} we have $k$ exact sequences 

\begin{equation}\label{eq:OrbExSeq1}
1\to \pi_1(X_i,x_i)\to \pi_1(\mathcal{X}_i,\bar{x}_i)\overset{\varphi_i}{\to} (G,\mbox{Id})\to 1
\end{equation}
where $\mathcal{X}_i=[X_i/G]$, $x_i\in X_i$ and its image in $\mathcal{X}_i$ is denoted by $\bar{x}_i$.

Denote by $\mathbb{H}:=\pi_1(\mathcal{X}_1,x_1)\times_G\cdots\times_G\pi_k(\mathcal{X}_k,x_k)$. The exact sequences in (\ref{eq:OrbExSeq1}) can be assembled as follows
\begin{equation} \label{eq:orbexseq3}
1\to \pi_1(X_1\times \cdots\times X_k,x)\to \mathbb{H}\to G\to 1 
\end{equation}
 with $x=(x_1,\ldots,x_k)$. The geometric nature of $\mathbb{H}$ is shown in the following lemma.
\begin{lem}\label{lem:OrbGroupIsFiberProduct} Let $G$ act diagonally on $X=X_1\times\cdots\times X_k$. Consider the stack $\mathcal{X}=[X/G]$ then  $\pi_1(\mathcal{X},\bar{x})\cong \mathbb{H}$.
\end{lem}
\begin{proof}
We have natural projection maps $\mathcal{X}\to \mathcal{X}_i$ for $i=1,\ldots,k$, which together with the morphisms $\varphi_i:\pi_1(\mathcal{X}_i,\bar{x}_i)\to G$ and the universal property of the fiber product give us a morphism $\pi_1(\mathcal{X},\bar{x})\to \mathbb{H}$. By the exact sequence of a fibration \S \ref{ss:FiberHomotopyES} applied to the action of $G$ to $X_1\times\cdots\times X_k$ and by (\ref{eq:orbexseq3}) we obtain 
\[
\begin{tikzcd}
1  \arrow[r] & \pi_1(X_1\times  \cdots \times X_k,x) \arrow[d,"\mathrm{id}"] \arrow[r] & \pi_1(\mathcal{X},\bar{x}) \arrow[d] \arrow[r] & G \arrow[r] \arrow[d,"\mathrm{id}"] & 1 \\
1 \arrow[r] & \pi_1(X_1\times  \cdots \times X_k,x) \arrow[r] & \mathbb{H} \arrow[r] & G \arrow[r]& 1
\end{tikzcd}
\]
which implies the result.
\end{proof}

\begin{lem} \label{lem:QuotientFixed} Let $X,X_i$ and $G$ be as above. Then $$\pi_1(X/G  ,[x]) \cong \pi_1(\mathcal{X},\bar{x})/N \cong \pi_1(\mathcal{X},\bar{x})/\mathbf{I} $$ 
where $N$ is the normal subgroup generated by the image of the inertia groups $I_x$ and $\mathbf{I}$ is the subgroup generated by the elements of $\pi_1(\mathcal{X})$ having  fixed points in the universal cover of $X_1\times \cdots \times X_k$.
\end{lem}
\begin{proof}
By \cite[Theorem~8.3(i)]{noohi2008} we have that  $\pi_1(X/G,[x])\cong \pi_1(\mathcal{X},\bar{x})/N$. The group $\pi_1(\mathcal{X},\bar{x})$ acts on $\tilde{X}\cong \tilde{X_1}\times \cdots \times\tilde{X}_k$ the universal cover of $X_1\times \cdots \times X_k$ in such a way that $[(\tilde{X}_1\times \ldots \times \tilde{X}_k)/\pi_1(\mathcal{X},\bar{x})]\cong \mathcal{X}$. As $G$ is finite, by Lemma \ref{lem:InertiaIsInjective} any stabilizer $I_x$ for the action of $\pi_1(\mathcal{X})$ on $\tilde{X}$ is finite, therefore it has the slice property and by \cite[Theorem~9.1]{noohi2008} we obtain that $\pi_1(X/G,[x])\cong \pi_1(\mathcal{X},\bar{x})/\mathbf{I}$ .

\end{proof}

%We will be interested in some particular kind of groups such that $\mathbf{I}$ can be explicitly described in the following way.

%\begin{defn} Let $H$ be a group. Denote by $\Tors(H)$ the normal subgroup of $H$ generated by its elements of finite order.
%\end{defn}

%Our main objects of study will satisfy $\mathbf{I}=\Tors(\pi_1(\mathcal{X}))$.

%\begin{exmp}
%Let $\Gamma$ be a Fuchsian group of the first kind ($\Gamma\backslash \mathcal{D}$ is of type $(g;n;0)$)  admitting a presentation:
%$$\left\langle \alpha_1,\ldots,\alpha_g, \beta_1,\ldots, \beta_g, \gamma_1,\ldots,\gamma_d,\rho_1,\ldots, \rho_n \mid \gamma_i^{m_i}, \prod [\alpha_i,\beta_i]\cdot \gamma_1\cdots\gamma_d\cdot\rho_1\cdots\rho_n  \right\rangle $$
%The quotient $\Gamma\backslash \mathcal{D}$ is a Riemann surface $C_{g,n}$ of genus $g$ with $n$ punctured points. 
%\end{exmp}
\section{The fundamental group of the product of topological spaces}\label{section:MainTheorem}
\subsection{Constructing the homomorphism}

\subsubsection{Finite index of the group in the product}\label{ss:FiniteIndexInTheProduct}
%As $G$ is finite and the action of $\pi_1(X)$ in $\tilde{X}$ is free, any the stabilizer $I_x$ for the action $\pi_1(\mathcal{X})$ over $\tilde{X}$ are finite then the mildness condition in [Foundations of topological stacks] is fulfilled, it has the slice property and then by \cite[Thm 9.1]{noohi2008}.

Let $I_y$ denote the isotropy at the point $y$ in $\tilde{X}$ for the action of $\pi_1(\mathcal{X},\bar{x})$. By Lemma \ref{lem:InertiaIsInjective} the map $\pi_1(\mathcal{X},\bar{x})\to G$ restricted to $I_y$ is injective, therefore we can identify $I_y$ with a subgroup of $G$. When we do such identification we will denote it by $I_y'<G$.

Now as $\pi_1(\mathcal{X},\bar{x})\cong\pi_1(\mathcal{X}_1,\bar{x}_1)\times_G\cdots\times_G\pi_1(\mathcal{X}_k,\bar{x}_k)$, if $y=(y_1,\ldots,y_k)$ we define $I_i<\pi_1(\mathcal{X}_i,\bar{x}_i)$ as the image of $I_y$ via the morphism $\pi_1(\mathcal{X},\bar{x})\to \pi_1(\mathcal{X}_i,\bar{x}_i)$.

\begin{lem}\label{lem:IsomorphismOfIsotropies} We have that $I_y\cong I_i$ for all $i=1,\ldots,k$ and  $I_y=I_1\times _{I_y'}\cdots \times_{I_y'} I_k$.
\end{lem}
\begin{proof}
For $\gamma=(\gamma_1,\ldots,\gamma_k)\in I_y$ note that $\gamma_i\in \pi_1(\mathcal{X}_i,\bar{x}_i)$ fixes $y_i\in\tilde{X}_i$, otherwise $\gamma$ can not fix a point in $\tilde{X}$. As above, the restriction of $\pi_1(\mathcal{X}_i,\bar{x}_i)\to G$ to $I_{y_i}$ is injective and as $I_i\subset I_{y_i}$ we have that $\gamma_i\not = \beta_i$ for $\gamma,\beta\in I_y\subset \pi_1(\mathcal{X}_1,\bar{x}_1)\times_G\cdots\times_G \pi_1(\mathcal{X}_k,\bar{x}_k)$ with $\gamma\not = \beta$. Therefore we can construct an inverse to the projection. The result follows.
\end{proof}

Note that we obtain that $I_i<I_{y_i}$, but in general $I_{y_i}$ can be bigger. Let us define the homomorphism $I_y\to \prod I_{y_i}$ given by decomposing an element in its components. By Lemma \ref{lem:IsomorphismOfIsotropies} it is injective.  Denote by $N$ the subgroup in $\pi_1(\mathcal{X},\bar{x})$ generated by all the $I_y$  and by $N_i '$ the subgroup in $\pi_1(\mathcal{X}_i,\bar{x}_i)$ generated by $I_i$. 

\begin{lem}\label{lem:N_iIsNormal} The subgroup $N_i'$ is normal in $\pi_1(\mathcal{X}_i,\bar{x}_i)$.
\end{lem}
\begin{proof}
Let $\gamma_i\in N_i'$ and $t_i\in \pi_1(\mathcal{X}_i,\bar{x}_i)$. We can write $\gamma_i=\gamma_{i_1}\cdots \gamma_{i_j}$ with each $\gamma_{i_l}\in I_{i_l}$ coming from $\gamma_l=(\gamma_{1_l},\ldots,\gamma_{i_l},\ldots, \gamma_{k_l})\in I_{y_l}\subset \pi_1(\mathcal{X},\bar{x})$ and the point $y_l=(y_{1_l},\ldots,y_{i_l},\ldots,y_{k_l}) \in \tilde{X} $ for $l=1,\ldots, j$. As every $\pi_1(\mathcal{X}_j,\bar{x}_j)\to G$ is surjective, for $j=1,\ldots,i-1,i+1,\ldots,k$, there exists $t_j\in \pi_1(\mathcal{X}_j,\bar{x}_j)$ such that $t=(t_1,\ldots,t_k)\in \pi_1(\mathcal{X},\bar{x})$.

 As $t \cdot \gamma_l \cdot  t^{-1}\in I_{ty_l }$ it follows that $t_i\gamma_{i_l}t_i^{-1}\in N_i'$ and therefore 
 $$t_i \gamma_i t_i^{-1}=(t_i\gamma_{i_1}t_i^{-1})\cdot  t_i\cdots t_i^{-1}\cdot (t_i\gamma_{i_j} t_i^{-1}) \in N_i'.$$
\end{proof}

%\begin{lem} Let $x\in\tilde{\mathcal{X}}$ such that $I_x$ is not trivial. Consider it as a subgroup of $G$. Then for $i=1,\ldots,k$ there exists $\bar{x}_i\in X_i$ such that $I_x \bar{x}_i=\bar{x}_i$. This is, if $I_{\bar{x}_i}$ denotes the isotropy at $\bar{x}_i$ for the action of $G$ in $X_i$ then $I_x<I_{\bar{x}_i}$
%\end{lem}
%\begin{proof}
%Write $x=(x_1,\ldots,x_k)$ with $x_i\in \tilde{X}_i$. Denote by $\bar{x}_i$ the image of $x_i$ in $X_i$. The action of $\pi_1(\mathcal{X}_i)$ over $\tilde{X}_i$ is $(\pi_1(\mathcal{X}_i)\to G)$-equivariant which implies the result.
%\end{proof}

\begin{prop}\label{prop:FiniteIndex} There is an homomorphism
\[\pi_1(X/G,[x])\to \prod_{i=1}^k \pi_1(\mathcal{X}_i,\bar{x}_i)/N_i'\]
such that the image has finite index.
\end{prop}
\begin{proof}
By Lemma \ref{lem:OrbGroupIsFiberProduct} we have that $\pi_1(\mathcal{X},\bar{x})\cong \pi_1(\mathcal{X}_1,\bar{x}_1)\times_G\cdots\times_G \pi_1(\mathcal{X}_k,\bar{x}_k)$. Therefore there is an injective homomorphism $\pi_1(\mathcal{X},\bar{x})\to \prod \pi_1(\mathcal{X}_i,\bar{x}_i)$. By Lemma \ref{lem:N_iIsNormal} we obtain the exact sequence \begin{equation}
1\to \prod_{i=1}^k N_i' \to \prod_{i=1}^k \pi_1(\mathcal{X}_i,\bar{x}_i)\to \prod_{i=1}^k \pi_1(\mathcal{X}_i,\bar{x}_i)/N_i'\to 1,
\end{equation}
and together with Lemma \ref{lem:QuotientFixed} we obtain a commutative diagram 
\begin{equation}\label{eq:SemiCommutativeD}
\begin{tikzcd}
  & 1 \arrow[d] & 1\arrow[d] \\
1 \arrow[r] & N \arrow[r] \arrow[d] & \prod_{i=1}^k N_i'  \arrow[d]\\
1 \arrow[r] & \pi_1(\mathcal{X},\bar{x}) \arrow[r] \arrow[d] & \prod_{i=1}^k \pi_1(\mathcal{X}_i,\bar{x}_i) \arrow[d] \\
 & \pi_1(X/G,[x])\arrow[d] \arrow[r]& \prod_{i=1}^k \pi_1(\mathcal{X}_i,\bar{x}_i)/N_i'\arrow[d]  \\
 & 1 & 1.
\end{tikzcd}
\end{equation}
This diagram provides a homomorphism
\[\pi_1(X/G,[x])\to \prod_{i=1}^k \pi_1(\mathcal{X}_i,\bar{x}_i)/N_i'\]
and shows that it is well defined. We can not complete (\ref{eq:SemiCommutativeD}) to a commutative diagram of groups with short exact sequence in the rows because usually $\pi_1(\mathcal{X},\bar{x})$ is not normal in $\prod_{i=1}^k \pi_1(\mathcal{X}_i,\bar{x}_i)$. It will be normal, for example, if $G$ is abelian. 

As $G$ is finite we obtain that $\pi_1(\mathcal{X},\bar{x})$ has finite index in $\prod_{i=1}^k \pi_1(\mathcal{X}_i,\bar{x}_i)$. In fact the upper-bound
\[\left[\prod_{i=1}^k \pi_1(\mathcal{X}_i,\bar{x}_i):\pi_1(\mathcal{X},\bar{x})\right]\leq \abs{G}^{k-1}\]
can be seen as follows. For each surjection $\varphi_i:\pi_1(\mathcal{X}_i,\bar{x}_i)\to G$ consider a lift $G_i\subset \pi_1(\mathcal{X}_i,\bar{x}_i)$ of $G$ with $\abs{G_i}=\abs{G}$. In $\prod_{i=1}^k G_i$, let us consider the equivalence relation
\[(g_1,\ldots, g_k)\sim (g_1',\ldots,g_k') \Leftrightarrow (\varphi_1(g_1),\ldots, \varphi_k(g_k))=(g\varphi_1(g_1'),\ldots,g\varphi_k(g_k'))\ \text{with}\ g\in G.\] 
It is easily seen that the quotient $\left(\prod_{i=1}^k G_i\right) /\sim \cong  (G\times \cdots \times G)/ \Delta_G$ is a set of representatives of left cosets $(\prod_{i=1}^k \pi_1(\mathcal{X}_i),\bar{x}_i) /\pi_1(\mathcal{X},\bar{x})$. By considering as coset representatives in $\prod_{i=1}^k \pi_1(\mathcal{X}_i,\bar{x}_i)/N_i'$ the image of $\prod_{i=1}^k G_i $ and using the diagram (\ref{eq:SemiCommutativeD}) we have that $\pi_1(X/G,[x])$ has finite index in $\prod_{i=1}^k \pi_1(\mathcal{X}_i,\bar{x}_i)/N_i'$.
\end{proof}

\subsection{The homomorphism has finite kernel}
\subsubsection{The subgroup $N_i'$ is finitely normally generated}\label{ss:DescriptionOfOrbispaces}
Let $X$ be a connected, semi-locally simply connected and locally path connected topological space. Let $G$ be a discrete finite group acting on $X$, $x\in X$ and denote by $\bar{x}\in \mathcal{X}=[X/G]$ the image of the point $x$ and by $p:X\to [X/G]$ the quotient map.

Let us briefly recall the description of $\pi_1(\mathcal{X},\bar{x})$ as given in \cite{chen2001homotopy}. It can be defined as $\pi_0(\Omega(\mathcal{X},\bar{x}))$ where $\Omega(\mathcal{X},\bar{x})$ denote the space loop of $\mathcal{X}$ pointed at the constant loop of value $\bar{x}$. Every loop is given locally as a map from an open subset of $S^1$ to a given uniformization of an open subset of $\mathcal{X}_\mathrm{top}$ and plus some gluing conditions. In our case of a global quotient, a more explicit description of $\Omega(\mathcal{X},\bar{x})$ can be given as follows.

Let $P(X,x)$ consist of paths in $X$ starting at $x$. As a subspace of $\Lambda(X)$, the free loop space of $X$, it inherits a structure of a topological space. By considering the constant loop $x$ of value $x\in X$, we obtain $(P(X,x),x)$ a pointed topological space. Define $P(X,G,x)$ as the subspace of $P(X,x)\times G$ consisting of the elements $(\gamma,g)$ satisfying $\gamma(1)=g \cdot \gamma(0)=g \cdot x$. As a topological space it is pointed at $(x, 1_G)$
\begin{lem}{\cite[Lemma 3.4.2]{chen2001homotopy}}
There exists a natural homeomorphism between the pointed topological spaces $(\Omega(\mathcal{X},\bar{x}),x)$ and $(P(X,G,x),(x,1_G))$.
\end{lem}

\begin{rem} When $(\mathcal{X},\bar{x})$ is a pointed topological stack there exists $(B[R\rightrightarrows X],x')$ a pointed topological space, where $B [R\rightrightarrows X]$ is the classifying space of the topological groupo\"id $[R \rightrightarrows X]$,  such that we can take $\pi_1(\mathcal{X},\bar{x}):=\pi_1(B[R\rightrightarrows X],x')$. In the case of a global quotient $\mathcal{X}=[X/G]$ it happens that $B[R\rightrightarrows X]$ equals the Borel construction $X\times_G EG$, see \cite{NOOHI20122014}.

Now, the construction of Chen also gives a natural isomorphism between $\pi_1(\mathcal{X},\bar{x})$ and $\pi_1(X\times_G EG, x')$ \cite[Theorem 3.4.1]{chen2001homotopy} linking both definitions.
\end{rem}

There exists a canonical projection $(P(X,G,x),(x,1_G))\to (G,1_G)$ given by sending $(\gamma,g)$ to $g$. This map can be seen to be a fibration \cite[Lemma 3.4.3]{chen2001homotopy} having as fiber at $1_G$ the space loop $\Omega(X,x)$ via the embedding $\Omega(X,x)\hookrightarrow P(X,G,x)$ where $\gamma$ maps to $(\gamma, 1_G)$.

\medskip

With this description at hand, suppose there is $y\in X$ such that it is fixed by an element $g$, that is, $y\in X^g$. Denote by $\gamma_y$ a path starting at $x$ and finishing at $y$, then $\gamma_y (g\gamma_y^{-1})\in P(X,G,x)$, where $g\gamma_y^{-1}$ denotes the action of $g$ applied to each point of the path.

\begin{lem} Let $I_{y}<G$ denote the inertia (stabilizer) of the action of $G$ at $y\in X$. Every homotopy class of a path $[\gamma_y]\in \pi_1(X,x,y)$ induces an injective morphism $I_y\to \pi_1(\mathcal{X},x)$.
\end{lem}
\begin{proof}
As $G$ is discrete $g\mapsto \gamma_g (g \gamma_g^{-1})$ is continuous, with $g\in I_y$. Then by taking the functor $\pi_0$ we got a morphism of groups $\pi_0(I_y)\to \pi_0(P(X,G,x))=\pi_1(\mathcal{X},\bar{x})$. Finally, by composing with the projection $(\pi_0(P(X,G,x),x))\to \pi_0((G,1_G))$ we obtain that different points under $\pi_0(I_y)\to\pi_1(\mathcal{X},\bar{x})\to \pi_0(G)$ have different images, thus the morphism in injective.
\end{proof}

\begin{lem}\label{lem:PathConnectedComponents}
Let $Y\in \pi_0(X^g)$, $y_1,y_2\in Y$ and let $\gamma_{y_1},\gamma_{y_2}$ be paths starting at $x\in X$ and finishing at $y_1$ and $y_2$ respectively, then $\gamma_{y_1} (g\gamma_{y_1}^{-1})$ is a conjugate of $\gamma_{y_2} (g\gamma_{y_2}^{-1})$ in $\pi_1(\mathcal{X},\bar{x})$ by elements of $\pi_1(X,x)$.
\end{lem}
\begin{proof}
There exists a path $\beta\subset Y$ connecting $y_1$ and $y_2$, therefore $\gamma_{y_1}\beta (g \beta^{-1}\gamma_{y_1}^{-1})\in P(X,G,x)$ but as $g\beta=\beta$ passing to $\pi_0(P(X,G,x),x)$ it equals $[\gamma_{y_1}(g \gamma_{y_1}^{-1})]$.

Now consider the path $\gamma_{y_2}$. Note that $\theta:=\gamma_{y_1}\beta \gamma_{y_2}^{-1}\in \Omega(X,x)$. There exists a continuous map
\[\#:P(X,G,x)\times P(X,G,x)\to P(X,G,x)\]
which induces the multiplication in the fundamental group (see \cite[Section~3.1]{chen2001homotopy}).  The element $\theta \#(\gamma_{y_2} (g \gamma_{y_2}^{-1})\#\theta^{-1})  $ can be seen to be $\theta ( \gamma_{y_2}\cdot (g(\theta \cdot \gamma_{y_2})^{-1}))\in P(X,G,x)$. By passing to the group $\pi_1(\mathcal{X},\bar{x})=\pi_0(P(X,G,x),x)$ we have that $ [\theta][\gamma_{y_2} (g \gamma_{y_2}^{-1})][\theta^{-1}]= [\gamma_{y_1}(g\gamma_{y_1}^{-1})]$.
\end{proof}

Recall that given $(X,x)$ as above, we have a pointed universal cover map $p:(\tilde{X},\tilde{x})\to (X,x)$ where $\tilde{x}$ represents the constant loop of value $x$. Every element in $\gamma\in\pi_1(X)$ corresponds to a point in $p^{-1}(x)$. So given a pointed map $p_\gamma:(\tilde{X},\gamma)\to (X,x)$ it induces a deck transformation of $\tilde{X}$ in the following way:
given $y\in \tilde{X}$ take a path $\alpha_y\subset \tilde{X}$ starting at $\gamma$ and finishing at $y$. Consider the unique lift $\tilde{p_\gamma(\alpha_y)}\subset \tilde{X}$ starting at $x$ and assign to $y$ the point $\tilde{p_\gamma(\alpha_y)}(1)$. It can be seen to be a well-defined map (see \cite{hatcher}). 

Now, by the description given above of $\pi_1(\mathcal{X},\bar{x})$, any $\gamma\in \pi_1(\mathcal{X},\bar{x})$ such that $\varphi(\gamma)=g$ (recall that $\varphi:\pi_1(\mathcal{X},\bar{x})\to G$) have as a representative an element in $P(X,G,x)$ which we still denote by $\gamma$. So $\gamma$ starts at $x$ and finishes at $gx$. Denote by $\tilde{\pi}:(\tilde{X},\tilde{x})\to (\mathcal{X},\bar{x})$ the universal cover morphism, note that $\tilde{\pi}_\gamma:(\tilde{X},\gamma)\to (\mathcal{X},\bar{x})$ is also a cover morphism. By \cite[Theorem~4.1.6]{chen2001homotopy} we obtain a deck transformation in the following way: given $y\in \tilde{X}$ take a path $\alpha_y\subset\tilde{X}$ starting at $\gamma$ and ending at $y$. Using the notation of the precedent paragraph, the path $p_\gamma(\alpha_y)$ starts at $gx$. Then the path $g^{-1} p_\gamma(\alpha_y)$ starts at $x$ so we can lift it to $\tilde{g^{-1}p_\gamma(\alpha_y)}$ in $(\tilde{X},\tilde{x})$, the end point of this lift is then defined as the image of $y$. It is shown that it is a well defined map and does not depend on the path chosen.

\begin{lem}\label{lem:everyElementFixingIsConjugate} Let $y\in X$ be fixed by $g\in G$, consider a path $\gamma_y$ connecting $x$ and $y$. Consider the action of $\pi_1(\mathcal{X},\bar{x})$ on $\tilde{X}$ given by deck transformations $\mathrm{Deck}(\tilde{X},\mathcal{X})$, then the element $\gamma_y(g\gamma_y^{-1})\in \pi_1(\mathcal{X},\bar{x})$ fixes a point in $\tilde{X}$. Moreover, any  element of $\pi_1(\mathcal{X},\bar{x})$ fixing a point in $\tilde{X}$ is of this form.
\end{lem}
\begin{proof}
As the endpoint of $\gamma_y(g\gamma_y^{-1})$ is $gx$ we have a pointed covering morphism
\[\tilde{\pi}_{\gamma_y(g\gamma_y^{-1})}:(\tilde{X},\gamma_y(g\gamma_y^{-1}))\to (\mathcal{X},\bar{x})\]
and we can consider $g\gamma_y$ as a path in $\tilde{X}$ connecting $\gamma_y(g\gamma_y^{-1})$ and $\gamma_y$ as follows: let us define $f(t)=\gamma(g \gamma_y^{-1})\cdot (g\gamma_y|_t)$ where $g\gamma_y|_t(t'):=g\gamma_y(t'/t)$ denote the path starting at $gx$ and finishing at $g\gamma_y(t)$ in time $t$ for $t\not =0$ and being the constant path with value $gx$ if $t=0$. We project then $f(t)$ to $X$ and obtain $g\gamma_y$ which starts at $gx$ and finishes at $\bar{y}$. By the discussion before the lemma, we obtain that it lifts to $\gamma_y$ in $(\tilde{X},\tilde{x})$, as $g$ fixes $y$ we obtain that the point $\gamma_y\in \tilde{X}$ is fixed by the induced deck transformation.

Consider the exact sequence 
$$1\to \pi_1(X,x)\to \pi_1(\mathcal{X},\bar{x})\overset{\varphi}{\to} G\to 1, $$
let $\gamma\in \pi_1(\mathcal{X},\bar{x})$ and $z\in \tilde{X}$ such that $\gamma$ fixes  $z$. Let $p:(\tilde{X},\tilde{x})\to (X,x)$ be the projection, as it is $\varphi$-invariant we have that $\varphi(\gamma) p(z)=p(z)$. Then by considering the path in $X$ corresponding to $z$, we can construct an element $z \varphi(\gamma)z^{-1}$, which fixes $z\in \tilde{X}$. As in the isotropy $\varphi$ is injective by Lemma \ref{lem:InertiaIsInjective}, we have that $z\varphi(\gamma)z^{-1}=\gamma$.
\end{proof}

\begin{prop}\label{prop:NiIsNormallyFinitelyGenerated} Suppose that there are only a finite number of elements in $\pi_0(X^g)$ for each $g\in G$, then  there exists a finite set $L\subset\pi_1(\mathcal{X},\bar{x})$ consisting of elements having fixed points in $\tilde{X}$ such that if $\gamma\in \pi_1(\mathcal{X},\bar{x})$ fixes a point in $\tilde{X}$ then it is conjugate to an element of $L$ by elements in $\pi_1(X,x)$.
\end{prop} 

\begin{proof}
By Lemma \ref{lem:PathConnectedComponents} for every element in $Y\in\pi_0(X^g)$ it suffices to fix an element $\gamma_y (g \gamma_y^{-1})$ with $y\in Y$. For every $g\in G$ and every element in $\pi_0(X^g)$ we pick such an element. We define $L$ as the set consisting of such elements. By Lemma \ref{lem:everyElementFixingIsConjugate} every such element fixes a point in $\tilde{X}$ and any other fixing a point will be conjugate of the element in $L$ corresponding to its connected component.
\end{proof}

\subsubsection{Proof that the homomorphism has finite kernel}
Let us return to the case of $k$-to\-po\-lo\-gi\-cal spaces $X_1,\ldots, X_k$  and let $G$ be a finite group acting on each one of them on the left as in \ref{ss:FundamentalQuoProd}. Proposition \ref{prop:NiIsNormallyFinitelyGenerated} gives us $k$ subsets $L(\mathcal{X}_i)\subset \pi_1(\mathcal{X}_i,\bar{x}_i)$  whose elements correspond to the element of $\pi_0(X_i^g)$ with $g\in G$. Now consider the subsets $L_i\subset L(\mathcal{X}_i)$ consisting of elements corresponding to $\pi_0(X_i^g)$ where $g$ fixes a point in $X_i$ for $i=1,\ldots,k$.

Recall that $N<\pi_1(\mathcal{X},\bar{x})$ (with $\mathcal{X}=[(X_1\times \cdots\times X_k)/G]$) is the subgroup generated by the inertia subgroups $I_y$ given by the action of $\pi_1(\mathcal{X},\bar{x})$ in $\tilde{X}$ and $N_i'<\pi_1(\mathcal{X}_i,\bar{x}_i)$ is the image of the $i$-projection of $N$. The following lemma is immediate from Proposition \ref{prop:NiIsNormallyFinitelyGenerated}

\begin{lem}\label{lem:LiGeneratesNi} We have that $N_i'=\left\langle \gamma_i l_i \gamma_i^{-1}\mid l_i\in L_i, \gamma_i\in \pi_1(X_i,x_i) \right\rangle$ in $\pi_1(\mathcal{X}_i,\bar{x}_i)$ for $i=1,\ldots,k$.
\end{lem}

\begin{defn} Let us define 
$$C_i=C_i(\pi_1(X_i),L_i):=\left\langle\left\langle \gamma_i l_i \gamma_i^{-1} l_i^{-1} \mid \gamma_i\in \pi_1(X_i,x_i), l_i\in L_i \right\rangle \right\rangle_{\pi_1(\mathcal{X}_i,\bar{x}_i)},  $$ to be the normal subgroup generated by the commutators of elements in $\pi_1(X_i,x_i)$ and in $L_i$. Denote by $\mathbb{T}_i:=\pi_1(\mathcal{X}_i,\bar{x}_i)/C_i$ and by $\hat{L}_i$ the image of $L_i$ in $\mathbb{T}_i$. 
\end{defn}

\begin{lem} It happens that $C_i<N_i'$ and moreover we can consider $C_i$ as a subgroup of $N$ via $\{e\}\times \cdots \times C_i \times \ldots \times \{e\}$ and $C_1\times \cdots \times C_k< N$.
\end{lem}
\begin{proof}
Let $l_i\in L_i$ and $\gamma_i\in\pi_1(X_i,x_i)$, the elements of $L_i$ were chosen such that there exists $l_j\in L_j$ and $y\in \tilde{X}$ such that $l=(l_1,\ldots,l_i,\ldots,l_k)\in I_y<N$. We have that $\gamma_i'=(e,\ldots,\gamma_i,\ldots,e)\in \pi_1(\mathcal{X},\bar{x})$ and as $N$ is normal in $\pi_1(\mathcal{X},\bar{x})$ we have that $\gamma_i' l \gamma_i'^{-1}\in N$, so 
$$\gamma_i'l \gamma_i'^{-1}l^{-1}=(e,\ldots,\gamma_i l_i \gamma_i^{-1} l_i^{-1},\cdots,e)\in N $$
This element projects to $[\gamma_i,l_i]\in C_i$. Finally given $\beta_i\in \pi_1(\mathcal{X}_i,\bar{x}_i)$, as every $\varphi_j$ is surjective, there exists $\beta_j\in \pi_1(\mathcal{X}_j,\bar{x}_j)$ such that $\varphi_i(\beta_i)=\varphi_j(\beta_j)$, so $\beta=(\beta_1,\ldots,\beta_k)\in \pi_1(\mathcal{X},\bar{x})$ and every conjugate of $[\gamma_i,l_i]$ can be seen as an element of $N$.

Finally, by considering the product of the identification of the elements in $C_j$ we have $C_1\times \cdots \times C_k< N$.
\end{proof}

Before stating the next lemma recall that $N<N_1'\times_G\cdots\times_G N_k'$.
\begin{lem}\label{lem:C_iHasFiniteIndexInN_i'}
The subgroup $C_i$ has finite index in $N_i'$, in particular $C_1\times \cdots\times C_k$ has finite index in $N_1'\times\cdots\times N_k'$ hence also in $N$. 
\end{lem}
\begin{proof}
First note that by Lemma \ref{lem:LiGeneratesNi} and by definition of $\mathbb{T}_i$ we have
\[N_i'/C_i=\left\langle \left\langle L_i\right\rangle \right\rangle_{\pi_1(X,x)}/C_i \cong \langle \langle \hat{L}_i \rangle\rangle_{R_i}=\langle \hat{L}_i\rangle,\]
with $R_i$ the image of $\pi_1(X_i,x_i)$ in $\mathbb{T}_i$.

 Moreover as $\varphi(C_i)=\{e\}$ we have that $C_i<\ker \varphi\cong\pi_1(X_i,x_i)$. As $\pi_1(X_i,x_i)$ has finite index in $\pi_1(\mathcal{X}_i,\bar{x}_i)$, it follows that $R_i$ has finite index in $\mathbb{T}_i$, which implies that $R_i\cap\langle\hat{L}_i \rangle$ has finite index in $\langle\hat{L}_i \rangle$. Note that $\langle\hat{L}_i \rangle$ is generated by a finite number of torsion elements and that by construction $R_i\cap\langle \hat{L}_i \rangle$ is a central group in $\langle \hat{L}_i \rangle$.  As any group generated by a finite number of torsion elements and such that the center has finite index is finite (see \cite[Lemma~4.6]{CataneseBauer1}) the result follows.
\end{proof}
 
\begin{thm}\label{thm:FiniteKernel} The homomorphism $\pi_1(X/G,[x])\to \prod_{i=1}^k \pi_1(\mathcal{X}_i,\bar{x}_i)/N_i'$ has finite kernel.
\end{thm}
\begin{proof}By composing the quotient map $\prod_{i=1}^k \pi_1(\mathcal{X}_i,\bar{x}_i)\to \prod_{i=1}^k \pi_1(\mathcal{X}_i,\bar{x}_i)/N_i'$ with the inclusion $\pi_1(\mathcal{X},\bar{x})\to \prod_{i=1}^k\pi_1(\mathcal{X}_i,\bar{x}_i)$ we obtain $\pi_1(\mathcal{X},\bar{x})\to \prod_{i=1}^k \pi_1(\mathcal{X}_i,\bar{x}_i)/N_i'$ with kernel $(N_1'\times\cdots\times N_k')\cap \pi_1(\mathcal{X},\bar{x})=N_1'\times_G\cdots\times_G N_k'$ by the description of $\pi_1(\mathcal{X},\bar{x})$ as fiber product. We put this as a row in the following commutative diagram together with a vertical column given by Lemma \ref{lem:QuotientFixed} and complete to
\begin{equation}
\begin{tikzcd}
  & 1 \arrow[d] & 1\arrow[d] &  \\
 & N \arrow[r,"\mathrm{id}"] \arrow[d] & N \arrow[d] & \\
1 \arrow[r] & N_1'\times_G\cdots\times_G N_k' \arrow[r] \arrow[d] & \pi_1(\mathcal{X},\bar{x}) \arrow[d] \arrow[r] & \prod_{i=1}^k \pi_1(\mathcal{X}_i,\bar{x}_i)/N_i' \arrow[d,"\mathrm{id}"]  \\
1 \arrow[r] &  N_1'\times_G\cdots\times_G N_k'/N \arrow[d] \arrow[r]& \pi_1(X/G,[x])\arrow[d] \arrow[r]  &  \prod_{i=1}^k \pi_1(\mathcal{X}_i,\bar{x}_i)/N_i' \\
 & 1 & 1 & 
\end{tikzcd}
\end{equation}
By Lemma \ref{lem:C_iHasFiniteIndexInN_i'} both $N_1'\times_G\cdots\times_G N_k'/C_1\times \cdots\times C_k$ and $N/C_1\times \cdots\times C_k$ are finite hence 
$$\frac{N_1'\times_G\cdots\times_G N_k'/C_1\times \cdots\times C_k}{N/C_1\times \cdots\times C_k}\cong N_1'\times_G\cdots\times_G N_k'/N   $$ is finite.

\end{proof}
\subsubsection{Geometric interpretation of the groups $\pi_1(\mathcal{X}_i,\bar{x}_i)/N_i'$}
Let us denote by $I$, the subgroup of $G$ generated by the elements having a fixed point in every $X_i$ for $i=1,\ldots, k$. Note that $I$ is a normal subgroup.

Let $x'_i$ denote the class of $x_i$ in $X/I$ and $\bar{x}'_i$ the image of $x'_i$ in $[(X_i/I)/(G/I)]$.
\begin{prop}\label{prop:Quotient} There is an isomorphism $$\pi_1(\mathcal{X}_i,\bar{x}_i)/N_i'\overset{\sim}{\longrightarrow} \pi_1([(X_i/I)\left /(G/I)\right.],\bar{x}'_i).$$ 
\end{prop}
 
\begin{proof}
Observe that the action of $G$ on $X_i$ descends to an action of $G/I$ on $X_i/I$ and therefore we can define $[(X_i/I)\left / (G/I)\right.]$. Recall by the previous subsection \ref{ss:DescriptionOfOrbispaces} that $\pi_1(\mathcal{X}_i,\bar{x}_i)$ can be identified with the set of path-components of $P(X_i,G,x)$. Therefore an element $[\gamma]\in \pi_1(\mathcal{X}_i,\bar{x}_i)$ can be represented by a path $\gamma$ in $X_i$ starting at $x_i$ and finishing at $gx_i$ for some $g\in G$. Denote by $p_i:X_i\to X_i/I$ the quotient map. By considering $p_i(\gamma)$, we obtain a morphism between $\pi_1(\mathcal{X}_i,\bar{x}_i)$ and $\pi_1([(X_i/I/G/I)],\bar{x}'_i)$. 

It is immediate to see that the paths coming from the inertia of $I$ in $X_i$, that is, the elements of the form $\gamma_y (g \gamma_y^{-1})$ with $g\in I$ and $y\in X_i^g$, are sent to the trivial element in $\pi_1(X_i/I,x_i')$.

Now consider $\gamma\in \ker\left(\pi(\mathcal{X}_i,\bar{x}_i)\to \pi_1([(X_i/I)\left/\right. (G/I) ],\bar{x}'_i)\right)$. Then $\gamma$ is represented by a path in $X_i$, which we still denote by $\gamma$, starting at $x_i$ and finishing at $gx_i$ with $g\in G$. Note that moreover $g\in I$, otherwise by the projection $\pi_1([X_i/I\left/ G/I \right.],\bar{x}'_i)\to G$ the element would be sent to a non-zero element. Hence the image of $\gamma$ lies in $\pi_1(X_i/I,x_i')$ and it is trivial. By the exact sequence 
$$1\longrightarrow N_{[X_i/I]}\longrightarrow \pi_1([X_i/I],\bar{x}_i)\longrightarrow \pi_1(X/I,x_i')\to 1 $$
and noticing that $N_{[X_i/I]}=N_i'$ we have that $\gamma\in N_i'$ which proves the result.
\end{proof}

\section{Applications}\label{section:Applications}

\subsection{Product of the same topological space} Now let us describe a case where $N_i'$ equals the whole subgroup $N_i$ generated by the elements having a fixed point in the universal cover.
\begin{cor} Let $X_i=X_1$ for $i=2,\ldots,k$ and $G$ finite acting on $X_1$. Then the morphism 
\[\pi_1\left((X_1\times \cdots\times X_1)/G,[x]\right)\longrightarrow \prod_{i=1}^k \pi_1(X_1/G,[x_i])\] has finite kernel.
\end{cor}
\begin{proof}
We only have to show that $N_1'=N_1$ and then we obtain the result by applying Theorem \ref{thm:FiniteKernel}. By construction we have that $N_1'\subset N_1$. Let us show the inverse inclusion. Take $\gamma_1\in N_1$, then we can write $\gamma_1=\gamma_{1_1}\cdots \gamma_{1_l}$ such that there exists $y_{1_j}\in \tilde{X}_1$ satisfying $\gamma_{1_j}\in I_{y_{1_j}}$ for $j=1,\ldots, l$. As $\tilde{X}=\tilde{X_1}\times\cdots\times \tilde{X}_k$ by taking $y_j=(y_{1_j},\ldots,y_{1_j})\in \tilde{X}$ we have that $\gamma^j=(\gamma_{1_j},\ldots,\gamma_{1_j})\in I_{y_j}$ and therefore $\gamma=\gamma^1\cdots \gamma^l \in N$ and the image of $\gamma$ in $N_1$ equals $\gamma_1$.
\end{proof}
Another proof using Proposition \ref{prop:Quotient} can be obtained as follows: the action of $G/I$ is free in $X_1/I$ and since $X_1/G \cong (X_1/I)\left / (G/I) \right.$ we have $\pi_1\left([(X_1/I)\left/ (G/I)\right.]\right)\cong\pi_1(X_1/G)$.
\subsection{Second Main Theorem}
\begin{thm}\label{thm:MainThm2}
Let $X_1,\ldots,X_k$ admit a universal cover and let $G$ be a finite group acting on each of them such that $\abs{\pi_0(X_i^g)}<+\infty$ for every $g\in G$ and $i=1,\ldots,k$. Denote $X=X_1\times
\cdots\times X_k$ and consider the diagonal action of $G$ on it. Suppose $\pi_1(X/G,[x])$ is residually finite, then $\pi_1(X/G,[x])$ has a normal finite-index subgroup $\mathcal{N}\cong H_1\times \cdots \times H_k$ isomorphic to a product of normal finite index subgroups subgroups $H_i\lhd \pi_1(X_i/I,[x_i])$.
\end{thm}
\begin{proof}
By Theorem \ref{thm:Int2} we get a morphism $\Theta:\pi_1(X/G,[x])\to \prod_{i=1}^k \pi_1([X_i/I/G/I])$ having finite kernel $E$. As $\pi_1(X/G,[x])$ is residually finite we can construct a finite-index normal subgroup $\Gamma \lhd \pi_1(X/G,[x])$ such that $\Gamma\cap E=\{e\}$.

The morphism $\Theta|_\Gamma:\Gamma\to \prod_{i=1}^k \pi_1([(X_i/I)\left/ (G/I) \right.],\bar{x}'_i)$ is therefore injective and moreover as the subgroup $\Theta(\pi_1(X/G))<\prod_{i=1}^k \pi_1([(X_i/I)\left/ (G/I) \right.],\bar{x}'_i)$ has finite index it follows that $\Theta(\Gamma)<\prod_{i=1}^k \pi_1([(X_i/I)\left/ (G/I) \right.],\bar{x}'_i)$ has finite index.

For every $i=1,\ldots,k$, we have $\pi_1(X_i/I,[x_i])<\pi_1([(X_i/I)\left/ (G/I)\right.],\bar{x}'_i)$ as a normal finite-index subgroup. Define the subgroup 
$$\Theta(\Gamma)_i:=\Theta(\Gamma)\cap \left(\{e_1\}\times \cdots\times \pi_1(X_i/I,[x_i])\times \cdots\times \{e_k\}\right) $$ where $e_k\in \pi_1(X_j/I,[x_j])$ is the identity element. As $\Theta(\Gamma)_i$ has finite index in $\pi_1(X_i/I,[x_i])$, there exists a normal subgroup of finite index $H_i$ of $\pi_1([(X_i/I)\left/ (G/I) \right.])$ contained in $\Theta(\Gamma)_i$.
Set $H:=H_1\times \cdots\times H_k$, then $H\lhd \Theta(\Gamma)$ and it is a finite-index normal subgroup of $\prod_{i=1}^k \pi_1([(X_i/I)\left/ (G/I) \right.],\bar{x'}_i)$. The subgroup $\mathcal{N}:=\Theta^{-1}(H)\cap \Gamma$ sa\-tis\-fies the stated properties.
\end{proof}

%Finally, for the case of curves we obtain a structure statement of the fundamental group. For this subsection we follow closely the arguments of \cite{CataneseBauer1} and \cite{dedieu}.

\subsubsection{Case of smooth curves}

\begin{cor}\label{cor:MainThmGroups} Let $C_1,\ldots, C_k$ be smooth algebraic curves and let $G$ be a finite group acting on each $C_i$. Denote $C=C_1\times \cdots\times C_k$. Consider $\mathcal{C}=[C/G]$ with $G$ acting diagonally on $C$. Then $\pi_1(C/G)$ has a normal subgroup $\mathcal{N}$ of finite index isomorphic to $\Pi_1\times \cdots \times\Pi_k$ where $\Pi_i$ is either a surface group or a finitely generated free group for $i=1,\ldots,k$.
\end{cor}

By Theorem \ref{thm:Int2} we have a morphism $\pi_1(C/G)\to \prod_{i=1}^k \pi_1([C_i/I\left / G/I\right.]) $ with finite kernel, however if the action of $G/I$ is not faithful on $C_i/I$ then $\pi_1([(C_i/I)\left / (G/I)\right.])$ is not necessarily an orbifold surface group. This can be overcome as follows: let $K_i:=\ker(G/I\to \Aut C_i/I)$ and $H_i:=(G/I)/K_i$. Denote by $\mathcal{C}_i:=[(C_i/I)/G/I]$ and by $\mathcal{C}_i':=[(C_i/I) /H_i]$, we have a canonical morphism $\mathcal{C}_i\to \mathcal{C}_i'$.

\begin{lem} The induced homomorphism $q_i:\pi_1(\mathcal{C}_i)\to \pi_1(\mathcal{C}_i')$ is surjective and has finite kernel.
\end{lem}
\begin{proof}
By choosing a point $x_i\in C_i$ and denoting by $\bar{x}_i$ its image in both $\mathcal{C}_i$ and $\mathcal{C}_i'$ we obtain a fibration $[\mathrm{pt}/K,\mathrm{pt}]\hookrightarrow (\mathcal{C}_i,\bar{x}_i)\to  (\mathcal{C}_i',\bar{x}_i)$. By taking the long homotopy exact sequence 
$$\ldots \to \pi_2(\mathcal{C}_i',\bar{x}_i)\to \pi_1(\mathrm{pt}/K,\mathrm{pt}) \to \pi_1(\mathcal{C}_i,\bar{x}_i)\to \pi_1(\mathcal{C}_i',\bar{x}_i)\to 1, $$
as there is an isomorphism between $\pi_1(\mathrm{pt}/K,\mathrm{pt})$ and $\pi_0(K,1_K)$, the result follows.
\end{proof}
So by composing, we obtain a morphism $\Theta:\pi_1(C/G)\to \prod_{i=1}^k \pi_1(\mathcal{C}_i)\to \prod_{i=1}^k \pi_1(\mathcal{C}_i')$, this allows us to prove the following lemma, which together with Theorem \ref{thm:MainThm2} will imply Corollary \ref{cor:MainThmGroups}.

\begin{lem}\label{lem:Ci'IsResiduallyFinite} The group $\pi_1(C/G)$ is residually finite.
\end{lem}
\begin{proof}
First note that, as $\pi_1(\mathcal{C}_i')$ is an orbifold surface group, it is in particular residually finite. Now, it follows that $\Theta(\pi_1(C/G))$ is residually finite as it is a finite-index subgroup of a direct product of residually finite groups.

We need another property of these groups to continue. Let $H$ be a group and let $\hat{H}$ be its profinite completion. A group $H$ is called \emph{good} if for each $k\geq 0$ and for each finite $H$-module $M$ the natural homomorphism 
$$H^k(\hat{H},M)\to H^k(H,M) $$
is an isomorphism. In \cite[Lemmas~3.2 and~3.4, Proposition~3.6]{grunewald2008} it is shown that a finite-index subgroup of a good group is good, the product of good groups is good and that $\pi_1(\mathcal{C})$ for $\mathcal{C}$ an algebraic orbifold curve is good. We obtain therefore that $\Theta(\pi_1(C/G))$ is good.

Finally, \cite[Proposition~6.1]{grunewald2008} asserts that if $T$ is a residually finite good group and $\varphi:H\to T$ is a surjective homomorphism with finite kernel then $H$ is residually finite.
Applying this to $\Theta':\pi_1(C/G)\to \Theta(\pi_1(C/G))$ we obtain the result.
\end{proof}

\subsection{Partial compactifications of arrangement of lines}
The original motivation of this work was to study the partial compactifications of the complement of an arrangement of lines in $\mathbb{P}_\mathbb{C}^2$ which is the topic of my Ph.D. thesis. In \cite{aguilar2019fundamental} a general method for computing a presentation of the fundamental group was given and some examples studied. A family of arrangements related to the studied in \emph{op.~cit.} is available for any $n\in \mathbb{N}$, however this will require a treatment one by one. The results obtained here can be used to study some of these partial compactifications in family.

\subsubsection{Partial compactification of the complement of an arrangement of lines\label{ss:LAC}}
Consider the projective plane $\mathbb{P}^2_\mathbb{C}$ with homogeneous coordinates $(z_1:z_2:z_3)$.

 Let $\A=\sum_{i=1}^k L_i$ be a divisor in $\mathbb{P}^2$ such that the irreducible components $L_i$ are copies of $\mathbb{P}^1$ (lines). Then the singular set $\Sing \A$ of $\A$ consists only of points. Consider $\pi: \Bl_{\Sing\A}\mathbb{P}^2\to \mathbb{P}^2$ the blow up of the projective plane at the points $\Sing \A$. The divisor $\pi^* \A=\sum_{i=1}^{k+\abs{\Sing \A}} D_i$ has as irreducible components copies of $\mathbb{P}^1$, with $D_1,\ldots, D_k$ being the strict transform of $L_1,\ldots,L_k$ respectively and $D_{k+1},\ldots, D_{k+\abs{\Sing \A}}$ being the exceptional divisors. Take a subset $J\subset \{1,2,\ldots,k+\abs{\mbox{Sing} \A}\}$. The surface $\Bl_{\Sing \A}\mathbb{P}^2\setminus (\cup_{j \in J} D_j)$ is called a \emph{partial compactification of} $\mathbb{P}^2\setminus (\cup L_i)$. We are interested in how the fundamental group changes when we partially compactify the complement of such an arrangement $\A\subset \mathbb{P}^2$.
\subsubsection{Examples}
 The subvariety of $\mathbb{P}^2$
$$\Ceva(n):=\{(z_1:z_2:z_3)\mid (z_1^n-z_2^n)(z_1^n-z_3^n)(z_2^n-z_3^n)=0\} $$
can be seen as the union of the closure of the three singular fibers of the rational map $f:\mathbb{P}^2\dashrightarrow \mathbb{P}^1$ given by $(z_1:z_2:z_3)\mapsto ((z_1^n-z_2^n):(z_2^n-z_3^n))$. The map $f$ is not defined in a subset $S=\{p_1,\ldots,p_{n^2}\}\subset \Sing \Ceva(n)$ consisting of $n^2$ points where $\A_1:=\{z_1^n-z_2^n=0\}=\sum_{i=1}^n L_i$ intersects $\A_2:=\{z_1^n-z_3^n=0\}=\sum_{i=n+1}^{2n}L_i$. It actually happens that $S\subset \A_3:=\{z_2^n-z_3^n=0\}=\sum_{i=2n+1}^{3n}L_i$ and $S$ consists of points where $3$ lines of $\Ceva(n)$ meet. We have another $3$ points $p_{n^2+i}$ in $\Sing( \Ceva(n))$ which correspond to each singular point of $\A_i$ for $i=1,2,3$ and hence of multiplicity $n$.

The rational map $f$ extends to a morphism $\tilde{f}:\Bl_{\Sing\Ceva(n)}\mathbb{P}^2\to \mathbb{P}^1$ having as generic fiber the \emph{Fermat curve of degree} $n$ defined as $F(n):=\{z_1^n+z_2^n+z_3^n=0\}\subset \mathbb{P}^2$. Therefore $\tilde{f}$ is an isotrivial fibration.

Denote by $\mu(n)$ the group of $n^\mathrm{th}$ roots of unity. By taking $3$ copies of it we define $H(n):=\mu_1(n)\oplus\mu_2(n)\oplus\mu_3(n)/\left\langle \mu_1\mu_2\mu_3=1 \right\rangle$ where $\mu_i\in \mu_i(n)$. It acts on $F(n)$ via $(z_1:z_2:z_3)\mapsto (\mu_1z_1:\mu_2 z_2:\mu_3 z_3)$. The proof of the following theorem will appear elsewhere.
\begin{thm}
Consider the diagonal action of $H(n)$ in $F(n)\times F(n)$. Denote by $S$ the minimal resolution of $F(n)\times F(n)/H(n)$.
\begin{enumerate}
\item The fibration $S\to (F(n)\times F(n))/H(n)\to F(n)/H(n)\cong \mathbb{P}^1$ is isomorphic to $\tilde{f}$.

\item Every singular point in $F(n)\times F(n)/H(n)$ corresponds to the contraction of the strict transform $D_i$ of some line $L_i\in \Ceva(n)$.
\item The contraction of the $n$ lines corresponding to $\A_i$ lie in the line $E_i$ which is the exceptional divisor corresponding to the unique singular point in $\A_i$.
\item $E_i$ is mapped to a point via $F(n)\times F(n)/H(n)\to \mathbb{P}^1$.
\end{enumerate}

\end{thm}

The Fermat curve $F(n)$ of degree $n$  can be seen as a branched covering of $\mathbb{P}^1$ of degree $n^2$ via the morphism in $\mathbb{P}^2$ given by $F(n)\ni(z_1:z_2:z_3)\to (z_1^n:z_2^n:z_3^n)\in \{w_0+w_1+w_2=0\}$ which branches at the points $(1:-1:0),(1:0:-1),(0:1:-1)$. Over each branching point there are $n$ points, we denote by $X_1,\ldots,X_n$ for those over $(1:-1:0)$,  by
$Y_1,\ldots,Y_n$ over $(1:0:-1)$ and $Z_1,\ldots, Z_n$ over $(1:-1:0)$.

Recall that for $S\to S'$ be a resolution of singularities of $S'$, if $S'$ has only quotient singularities, by \cite[Theorem~7.8.1]{Kollar1993} we have that $\pi_1(S)\to\pi_1(S')$ is an isomorphism.

\begin{exmp}\label{ex:4.1} Consider the surface $S_1:=\big(F(n)\times( F(n)\setminus \{X_1,\ldots, X_n\})\big)\,/H(n)$. The subgroup $I$ generated by the elements of $H(n)$ having fixed points both in $F(n)$ and in $F(n)\setminus \{X_1,\ldots, X_n\})$ equals $H(n)$. As $F(n)/H(n)\cong \mathbb{P}^1$, $F(n)\setminus \{X_1,\ldots, X_n\})/H(n)\cong \mathbb{C}$ and by Theorem \ref{thm:Int2} the morphism 
$$\pi_1(S_1)\to \pi_1(\mathbb{P}^1)\times \pi_1(\mathbb{C}) $$
has finite kernel, it follows that $\pi_1(S_1)$ is finite. 

The minimal resolution of singularities $S_1'\to S_1$ can be identified with the fo\-llo\-wing partial compactification of $\Ceva(n)$. Consider $$J:=\{1,\ldots,n, 3n+n^3+1\}\subset \{1,\ldots, 3n +n^2+3\}$$
then following the construction given in \S \ref{ss:LAC} we have that $$\Bl_{\Sing \Ceva(n)}\mathbb{P}^2\setminus \{\cup_{j \in J}D_j\}\cong S_1'.$$
That is, from the surface $\Bl_{\Sing \Ceva(n)}\mathbb{P}^2$ we remove only the strict transform of $\A_1$ and the exceptional divisor coming from the singular point of $\A_1$. This can be identified with a singular fiber or $\tilde{f}$.
\end{exmp}

\begin{exmp}
Consider now $S_2:=(F(n)\times F(n)\setminus\{X_i,Y_i\})/H(n)$. In this case the subgroup $I$, defined as in the previous paragraph, is isomorphic to $\mu(n)$. As $F(n)/\mu(n)\cong \mathbb{P}^1$, $F(n)\setminus \{X_i,Y_i\})/\mu(n)\cong \mathbb{C}^*$ and by Theorem \ref{thm:Int2} the morphism 
$$\pi_1(S_2)\to \pi_1([\mathbb{P}^1/\mu(n)])\times \pi_1([\mathbb{C}^*/\mu(n)]) $$
has finite kernel and the image is a finite-index subgroup.

By Theorem \ref{thm:MainThm2} and Corollary \ref{cor:MainThmGroups}, we have that $\mathbb{Z}\lhd \pi_1(S_2)$ has finite index. As in Example \ref{ex:4.1} the minimal resolution of singularities $S_2'\to S_2$ can be identified with $\Bl_{\Sing \Ceva(n)}\mathbb{P}^2$ minus two singular fibers of $\tilde{f}$.
\end{exmp}

\begin{exmp} If we consider $S_3:=(F(n)\times F(n)\setminus \{X_i,Y_i,Z_i\})/H(n)$ it can be identified with $\Bl_{\Sing \A}\mathbb{P}^2$ minus the three singular fibers of $\tilde{f}$. As $H(n)$ acts freely in $F(n)\times F(n)\setminus \{X_i,Y_i,Z_i\}$, the long exact sequence of homotopy associated to the covering map $F(n)\times F(n)\setminus\{X_i,Y_i,Z_i\}\to S_3$ yields 
$$1\longrightarrow \pi_1(F(n))\times \pi_1(F(n)\setminus\{X_i,Y_i,Z_i\})\longrightarrow \pi_1(S_3)\longrightarrow H(n)\to 1. $$
\end{exmp}

\begin{rem} We can remove points also in the first component $F(n)$ of the product. However, we can not get more partial compactifications of $\Ceva(n)$ in this way. This can be shown by drawing the dual graph of the divisor $\pi^*\Ceva (n)$ and noticing that the lines obtained by removing points does not satisfy the intersection pattern of the graph. 
\end{rem}

%%%%%%%%%%%%%%%%%%%%%
% References
%%%%%%%%%%%%%%%%%%%%%

\ifx\undefined\bysame
\newcommand{\bysame}{\leavevmode\hbox to3em{\hrulefill}\,}
\fi


\begin{thebibliography}{BCGP12++}

\bibitem[Agu19]{aguilar2019fundamental}
R.~Aguilar Aguilar, {\em The fundamental group of partial compactifications of the
  complement of a real line arrangement}, preprint \arXiv{1904.11222} (2019).

\bibitem[BCGP12]{CataneseBauer1}
I.~Bauer, F.~Catanese, F.~Grunewald, and R.~Pignatelli, {\em
  Quotients of products of curves, new surfaces with $p_g = 0$ and their
  fundamental groups}, Am. J. of Math. {\bf 134} (2012), 993--1049.

\bibitem[Che01]{chen2001homotopy}
W.~Chen, {\em A Homotopy Theory of Orbispaces}, preprint \arXiv{math/0102020} (2001).

\bibitem[DP12]{dedieu}
T.~Dedieu and F.~Perroni, {\em The fundamental group of a quotient of a
  product of curves}, J. Group Theory {\bf 15} (2012), no.~3, 439--453.

\bibitem[GJZZ08]{grunewald2008}
F.~Grunewald, A.~Jaikin-Zapirain, and P.~A. Zalesskii, {\em Cohomological
  goodness and the profinite completion of Bianchi groups}, Duke Math. J. {\bf
  144} (2008), no.~1, 53--72.

\bibitem[Hat00]{hatcher}
A.~Hatcher, {\em {Algebraic topology}}, Cambridge Univ. Press, Cambridge, 2000.

\bibitem[Kol93]{Kollar1993}
J.~Kollár, {\em Shafarevich maps and plurigenera of algebraic varieties.},
  Invent. Math. {\bf 113} (1993), no.~1, 177--216.

\bibitem[{Noo}05]{2005noohi}
B.~{Noohi}, {\em {Foundations of Topological Stacks I}}, preprint \arXiv{math/0503247} (2005).

\bibitem[Noo08]{noohi2008}
\bysame, {\em Fundamental groups of topological stacks with the slice
  property}, Algebr. Geom. Topol. {\bf 8} (2008), no.~3, 1333--1370.

\bibitem[Noo12]{NOOHI20122014}
\bysame, {\em Homotopy types of topological stacks}, Advances in Math. {\bf 230} (2012), no.~4, 2014--2047.

\bibitem[Noo14]{NOOHI2014612}
\bysame, {\em Fibrations of topological stacks}, Advances in Math.
  {\bf 252} (2014), 612--640.

\end{thebibliography}
\end{document}